\documentclass[11pt]{amsart}
\usepackage{amssymb, amsmath, amsfonts, amsthm, esint,centernot,graphicx}
\usepackage[abbrev]{amsrefs}
\usepackage[margin=0.8 in]{geometry}
\usepackage[colorlinks=true, linkcolor=blue]{hyperref}

\usepackage{asymptote}

\newcommand{\NN}{\mathbb N}
\newcommand{\RR}{\mathbb R}
\newcommand{\ZZ}{\mathbb Z}

\numberwithin{equation}{section}

\def \I{\text{I}}
\def \II{\text{II}}
\def \III{\text{III}}
\def \IV{\text{IV}}
\def \V{\text{V}}

\def \dd {\partial}

\def \eps {\varepsilon}

\DeclareMathOperator{\id}{id}

\title{Fundamental gaps of spherical Triangles }
\author{Shoo Seto}
\address[Shoo Seto]{Department of Mathematics\\
	 California State University\\
Fullerton, CA 92831}
\email{\href{mailto:shoseto@fullerton.edu}{shoseto@fullerton.edu}}
\author{Guofang Wei}
\address[Guofang Wei]{Department of Mathematics\\
	University of California\\
	Santa Barbara, CA 93106}
\email{\href{mailto:wei@math.ucsb.edu}{wei@math.ucsb.edu}}
\thanks{The second author is partially supported by NSF DMS-1811558}
\author{Xuwen Zhu}
\address[Xuwen Zhu]{Department of Mathematics\\
Northeastern 	University\\
Boston, MA 02115}
\email{\href{mailto:x.zhu@northeastern.edu}{x.zhu@northeastern.edu}}
\thanks{The third author is partially supported by NSF DMS-2041823}
\date{}

\theoremstyle{definition}

\newtheorem{lemma}{Lemma}[section]
\newtheorem{remark}{Remark}[section]

\newtheorem{proposition}{Proposition}[section]
\newtheorem{theorem}{Theorem}[section]
\newtheorem{corollary}{Corollary}[section]

\begin{document}
	\begin{abstract}
		We compute Dirichlet eigenvalues and eigenfunctions explicitly for spherical lunes and the spherical triangles which are half the lunes, and show that the fundamental gap goes to infinity when the angle of the lune goes to zero. Then we show the spherical equilateral triangle of diameter $\frac{\pi}{2}$ is a strict local minimizer of the fundamental gap on the 
		space of the spherical triangles	with diameter  $\frac{\pi}{2}$, which partially extends Lu-Rowlett's result \cite{lu-rowlett} from the plane to the sphere.
	\end{abstract}
\maketitle

\section{Introduction}
	Given a bounded smooth domain $\Omega \subset M^n$ of a Riemannian manifold, the eigenvalue equation of the Laplacian on $\Omega$ with Dirichlet boundary condition  is
\begin{equation} \label{eve}
\Delta \phi = -\lambda\, \phi, \ \  \phi|_{\partial \Omega} =0. 
\end{equation}
 The eigenvalues consist of an infinite sequence going off to infinity. Indeed, the eigenvalues satisfy
$$   0 < \lambda_1 <  \lambda_2  \le \lambda_3  \cdots \to \infty. $$ In quantum physics the eigenvalues are possible allowed energy values and the eigenvectors are the quantum states which correspond to those energy levels.

	The fundamental (or mass) gap refers to the difference between the first two eigenvalues
\begin{equation} \label{fg} \Gamma (\Omega) = 
	\lambda_2 - \lambda_1 >0
\end{equation}
 of the Laplacian or more generally for Schr\"{o}dinger operators. It is a very interesting quantity both in mathematics and physics, and has been an active area of research recently. 

 In 2011, Andrews and Clutterbuck \cite{fundamental} proved the fundamental gap conjecture: for convex domains $\Omega \subset \mathbb R^n$ with diameter $D$, 
 $$\Gamma(\Omega)  \ge 3\pi^2/D^2.$$  
The result is sharp, with the limiting case being rectangles that collapse to a segment.    We refer to their paper for the history and earlier works on this important subject, see also the survey article \cite{daifundamental}. 

Recently, Dai, He, Seto, Wang, and Wei (in various subsets) \cite{seto2019sharp,dai2018fundamental,he2017fundamental} generalized  the  estimate to convex domains in $\mathbb{S}^n$, showing that the same bound holds:  $\lambda_2 - \lambda_1 \ge 3\pi^2/D^2$. Very recently, the second author with coauthors \cite{BCNSWW} showed the surprising result that there is no lower bound on the fundamental gap of convex domain in the hyperbolic space with arbitrary fixed diameter. This is done by estimating the fundamental gap of some suitable convex thin strips. 

For specific convex domains, one expects that the lower bound is larger. 
For triangles in $\mathbb R^2$ with diameter $D$, Lu-Rowlett \cite{lu-rowlett} showed that the fundamental gap is $\ge \frac{64 \pi^2}{9D^2}$ and equality holds if and only if it is an equilateral triangle. Note that explicit computation of the eigenvalues in general is very hard. For triangles the eigenvalues of only three types (the equilateral triangle and the two special right triangles) can be computed explicitly.   

In this paper we study some corresponding questions on the sphere.
First we compute the eigenvalues and eigenfunctions of the spherical lune  $L_\beta$ with angle $\beta$ which is the area bounded between two geodesics, see Figure~\ref{fig:lune}.  The statement about the eigenvalues is given below, while the eigenfunctions are included in the proof in Section 2. 
\begin{proposition}\label{2.1}
	The eigenvalues of Dirichlet Laplacian of the spherical lunes  $L_\beta$, 
	 without counting multiplicities, are given by the set
	$$
	\left\{\left(\frac{k\pi}{\beta}+j\right)\left(\frac{k\pi}{\beta}+j+1\right): k\in \NN^{+}, j\in \NN\right\}.
	$$	
	In particular, the first eigenvalue is $\frac{\pi}{\beta}(\frac{\pi}{\beta}+1)$, the fundamental gap is given by 
	$$
	3 \left(\frac{\pi}{\beta}\right)^{2}+\frac{\pi}{\beta}, \text{ if }\beta>\pi; \ \ \ 2\frac{\pi}{\beta}+2, \text { if } \beta\leq \pi. 
	$$
\end{proposition}
This is proven through separation of variables and analyzing the solutions to associated Legendre equations \eqref{e1}. Furthermore,  we  derive the eigenvalues  and eigenfunctions of the isosceles triangle which is half of the lune, whose eigenvalues are a subset of the ones of the lune.
\begin{proposition}  \label{prop2.2}
	For a spherical triangle with angles $\beta, \pi/2$ and $\pi/2$, its eigenvalues are given by 
	$$
	\left\{
	\left(\frac{k\pi}{\beta}+2j+1\right)\left(\frac{k\pi}{\beta}+2j+2\right): k\in \NN^{+}, j\in \NN
	\right\}.
	$$

	In particular, the first eigenvalue is $(\frac{\pi}{\beta}+1)(\frac{\pi}{\beta}+2)$, the fundamental gap is given by 
	\begin{equation}
	3 \left(\frac{\pi}{\beta}\right)^{2}+ 3\frac{\pi}{\beta}, \text{ if }\beta>\frac{\pi}{2}; \ \ \ 4\frac{\pi}{\beta}+10, \text { if } \beta\leq \frac{\pi}{2}.  \label{beta-gap}
	\end{equation}
\end{proposition}

On the plane, all equilateral triangles are related by scaling. On the other hand two equilateral triangles on the sphere are not conformal to each other. We concentrate on the $90^\circ$ equilateral triangle as for this one the eigenvalues and eigenfunctions can be computed explicitly. 
Analogously to \cite{lu-rowlett}, we  show that a spherical equilateral triangle of diameter $\frac{\pi}{2}$ is a local minimizer of the fundamental gap. 
\begin{theorem} \label{gap-tri}
The equilateral spherical triangle with angle $\frac{\pi}{2}$ is a strict local minimum for the gap  on the 
 space of the spherical triangles	with diameter  $\frac{\pi}{2}$.  Moreover \[
 \lambda_2 (T(t)) - \lambda_1(T(t)) \ge 18 +  \frac{16}{\pi}t + O(t^2),   \]
 where $T(t)$ is the triangle with vertices $(0,0)$, $(\frac{\pi}{2},0)$ and $(\frac{\pi}{2}-bt,\frac{\pi}{2}-at)$ with $a^2+b^2=1, a \ge 0, \ b\ge 0$ under geodesic polar coordinates centered at the north pole.
\end{theorem}
To get the estimate we compute and estimate the first derivative of the first two eigenvalues at $t=0$ as in \cite{lu-rowlett}. For this we
construct a diffeomorphism $F_t$ which maps the triangle $T(0)$ to the triangle $T(t)$ to pull back the metric on $T(t)$ to the fixed triangle $T(0)$. Unlike in the plane case, the diffeomorphism $F_t$ here is nonlinear, which makes the computations quite involved. The proof is given in Section 3. To keep the idea clear we put a large part of the computation in the appendix.

We expect the results in this paper will be very useful for further study of the fundamental gap of convex domains of $\mathbb S^n$. 

{\bf Acknowledgments}: The first two authors would like to thank Zhiqin Lu and Ben Andrews for their interest in the work and helpful discussions. 
\section{Eigenvalues of spherical lunes and the equilateral triangle}
In this section we compute Dirichlet eigenvalues and eigenfunctions for the spherical lunes and a family of spherical triangles, proving Proposition~\ref{2.1} and \ref{prop2.2}. The computation is done through separation of variables and looking at solutions to associated Legendre equations. 

\subsection{Spherical lunes}
Consider a lune  of angle $\beta$ ($0<\beta<2\pi$) on a sphere, $L_\beta$ (see Figure ~\ref{fig:lune}),  which is the area between two meridians each connecting the north pole and south pole and forming an angle $\beta$. Take $(r,\theta)$ to be the geodesic polar coordinates centered at the north pole, then the spherical metric is given by
$$
g= dr^{2}+\sin^{2}r d\theta^{2}, \ 0\leq r\leq \pi,\  0\leq \theta\leq \beta.
$$

\begin{center}
\begin{figure}
\includegraphics[scale=0.8]{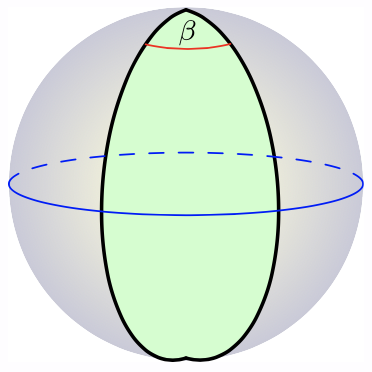}
\caption{A spherical lune of angle $\beta$}
\label{fig:lune}
\end{figure}
\end{center}
The Laplacian associated to this metric is given by
\begin{equation}
\Delta u(r,\theta)=\partial_{r}^{2}u + \frac{\cos r}{\sin r}\partial_{r}u +\frac{1}{\sin^{2} r} \partial_{\theta}^{2}u.  \label{lap}
\end{equation}
Hence the Dirichlet eigenvalue problem $\Delta u +\lambda u=0$ becomes
$$
\partial_{r}^{2}u + \frac{\cos r}{\sin r}\partial_{r}u +\frac{1}{\sin^{2} r} \partial_{\theta}^{2}u +\lambda u=0, \ u(r,0)=u(r,\beta)=0.
$$
Now we use separation of variables, and write $u(r,\theta)=R(r)\Theta(\theta)$. The boundary condition requires that 
$$\Theta(\theta)=\sin(\tfrac{k\pi}{\beta}\theta),\  k=1,2,\dots.$$
and correspondingly $R(r)$ satisfies 
\begin{equation} \label{eq-R}
R''(r)+\frac{\cos r}{\sin r}R'(r) - \frac{k^{2}\pi^{2}}{\beta^{2}\sin^{2} r}R(r) +\lambda R(r)=0.
\end{equation}

Taking $x=\cos (r)$, this becomes an equation for $R(x), -1\leq x\leq 1$:
\begin{equation}\label{e1}
(1-x^{2})R_{xx} - 2xR_{x}+[\lambda-\frac{k^{2}\pi^{2}}{\beta^{2}(1-x^{2})}]R=0.
\end{equation}
The equation is called \emph{general Legendre equation}~\cite[Ch.8]{AbramowitzStegun},
and there are two kinds of solutions to this equation, called the \emph{associated Legendre functions} $P_{\ell}^{\mu}$ (first kind) and $Q_{\ell}^{\mu}$ (second kind) of degree $\ell$ and order $\mu$, with 
\begin{equation}\label{e:initialellmu}
\lambda=\ell(\ell+1), \ \ell \in \RR; \ \mu=\pm \frac{k\pi}{\beta}.
\end{equation}
 
 With this set up, we are ready to prove Proposition \ref{2.1}. 
In general those functions have singularities at $-1$ and 1, so we need to find out those specific values of $\lambda$ such that for given $\mu$, the boundary conditions $R(-1)=R(1)=0$ are satisfied. 


\begin{proof}[Proof of Proposition~\ref{2.1}]
	To check the boundary condition we look at the asymptotics of $P_{*},\ Q_{*}$ at $1^{-}$,  and relations between $P_{*}(\pm x),\ Q_{*}(\pm x)$ to determine what combinations of $\ell$ and $\mu$ are admissible.

The asymptotic of $Q_{\ell}^{\mu}(x)$ as $x\rightarrow 1^{-}$~\cite[(14.8.3)-(14,8.6)]{1} and relations between $Q_{\ell}^{\pm\mu}(\pm x)$~\cite[(14.9.8)]{1} gives that the only possible case where $Q_{\ell}^{\mu}$ is regular will occur when $\mu=\frac{2N+1}{2}, N\in \ZZ$.
However in this case from~\cite[(14.2.6)]{1}, $Q_{\ell}^{\mu}$ and $P_{\ell}^{-\mu}$ are linearly dependent. 
Therefore we will only look at $P_{\ell}^{\mu}$ with $\ell,\mu$ given by~\eqref{e:initialellmu}.
In particular, from~\cite[(14.8.1)--(14.8.2)]{1} we have the two asymptotics
\begin{align}
&\lim_{x\rightarrow 1^{-}}P_{\ell}^{\mu}(x)=\frac{1}{\Gamma(1-\mu)}\left(\frac{2}{1-x}\right)^{\mu/2}, \ \mu\notin \NN^{+} \label{e:5}\\
&\lim_{x\rightarrow 1^{-}}P_{\ell}^{\mu}(x)=(-1)^{\mu}\frac{\Gamma(\ell+\mu+1)}{\Gamma(\ell-\mu+1)\mu!} \left(\frac{1-x}{2}\right)^{\mu/2}, \ \mu \in \NN^{+}, \ \ell\neq -\mu, -\mu+1, \dots, \mu-1 \label{e:6}
\end{align}
which restricts the possible choice to
\begin{enumerate}
	\item $\mu<0, \mu \notin \ZZ$;
	\item $\mu\in \ZZ$.
\end{enumerate}
We discuss those two cases separately. 

\noindent\textbf{Case 1: $\mu<0, \mu\notin \ZZ$.} The relation of $P_{*}(\pm x)$~\cite[(14.9.7)]{1} gives
\begin{equation}\label{e:14.9.7}
\frac{\sin(-\mu \pi)}{\Gamma(\ell+\mu+1)}P_{\ell}^{\mu}(-x)=-\frac{\sin\left((\ell+\mu)\pi\right)}{\Gamma(\ell-\mu+1)}P_{\ell}^{-\mu}(x) +\frac{\sin(\ell \pi)}{\Gamma(\ell+\mu+1)}P_{\ell}^{\mu}(x)
\end{equation}
or
\begin{equation}
\sin(-\mu \pi)P_{\ell}^{\mu}(-x)=-\frac{\Gamma(\ell+\mu+1)\sin\left((\ell+\mu)\pi\right)}{\Gamma(\ell-\mu+1)}P_{\ell}^{-\mu}(x) +\sin(\ell \pi)P_{\ell}^{\mu}(x)
\end{equation}

Take $x\rightarrow 1^{-}$ hence the left hand side $-x\rightarrow (-1)^{+}$, then the second term on the right side vanishes under the condition $\mu<0$. In order to have $\lim_{x\rightarrow 1^{-}}P_{\ell}^{\mu}(-x)=0$ we need the first term of the right hand side to vanish. Recall $P_{\ell}^{-\mu}(x\rightarrow 1^{-})=\infty$ from~\eqref{e:5}, so we need the coefficient of the first term to vanish. This gives two choices: 

(a) $\ell+\mu\in \ZZ, \ \frac{\Gamma(\ell+\mu+1)}{\Gamma(\ell-\mu+1)}\neq \infty$, 

 (b) $\Gamma(\ell+\mu+1)\neq \infty, \Gamma(\ell-\mu+1)=\infty$. 

Recall that $\Gamma(t)=\infty$ if and only if $t=0,-1,-2, \dots$. Case (a) and (b) above might overlap when $\mu=-\frac{1}{2}, -\frac{3}{2}, \dots$ and $\ell\pm \mu+1$ are integers. In this case we need  $\frac{\Gamma(\ell+\mu+1)}{\Gamma(\ell-\mu+1)}$ to be bounded. Since $\ell\pm\mu+1$  are both integers and $\ell-\mu+1>\ell+\mu+1$, we need $\ell-\mu+1\leq 0$ or $\ell+\mu+1> 0$. This leads to
$$
\begin{aligned}
\mu&=-j-\frac{1}{2}, \ j=0, 1,2,\dots, \\
\ell&=\mu-1,\mu-2, \dots \text{ or } -\mu, -\mu+1, -\mu+2, \dots.
\end{aligned}
$$

Now we assume $\mu$ is not a half-integer. For case (a), assuming $\ell+\mu\in \ZZ$, we only need to require $\Gamma(\ell+\mu+1)\neq \infty$ (since $\ell-\mu+1$ is not an integer, $\Gamma(\ell-\mu+1)$ is finite), hence 
$$\ell+\mu+1=1,2,3,\dots.$$ 
For case (b) we have 
$$\ell-\mu+1=0, -1,-2,\dots,$$ 
and it automatically satisfies $\Gamma(\ell+\mu+1)\neq \infty$. 

Combining the discussion above we get the admissible combination:
\begin{equation*}
\begin{aligned}
&\mu\notin \ZZ, \mu<0 \\
&\ell=\mu-1,\mu-2, \dots \text{ or } -\mu, -\mu+1, -\mu+2, \dots.
\end{aligned}
\end{equation*}
Notice that the above two choices of $\ell$ actually only give one eigenfunction since $P_{-\ell-1}^{\mu}(x)=P_{\ell}^{\mu}(x)$ (cf. ~\cite[(14.9.5)]{1}), therefore the admissible combination of $\ell$ and $\mu$ in this case can be written as
\begin{equation}
\begin{aligned}\label{e:ellmu2}
&\mu\notin \ZZ, \mu<0 \\
&\ell=|\mu|, |\mu|+1, |\mu|+2, \dots.
\end{aligned}
\end{equation}
Putting in $\mu=-\frac{k\pi}{\beta}$, the eigenvalues $\ell(\ell+1)$ are given by 
\begin{equation}\label{e:eigenvalue2}
\lambda=(\tfrac{k\pi}{\beta}+j)(\tfrac{k\pi}{\beta}+j+1), \ k\in \NN^{+}, \ j\in \NN, \ \frac{k\pi}{\beta}\notin \ZZ,
\end{equation}
and the corresponding eigenfunction is given by
$$
u=P_{\frac{k\pi}{\beta}+j}^{-\frac{k\pi}{\beta}} \left(\cos (r)\right) \sin (\tfrac{k\pi}{\beta}\theta).
$$

\noindent\textbf{Case 2: $\mu\in \ZZ$. } In this case we can similarly look at the asymptotic behavior of $P_{\ell}^{\mu}(x)$ as $x\rightarrow (-1)^{+}$ by using relation~\eqref{e:14.9.7}. 
$$
\frac{\sin(-\mu \pi)}{\Gamma(\ell+\mu+1)}P_{\ell}^{\mu}(-x)=-\frac{\sin(\ell+\mu)\pi)}{\Gamma(\ell-\mu+1)}P_{\ell}^{-\mu}(x) +\frac{\sin(\ell \pi)}{\Gamma(\ell+\mu+1)}P_{\ell}^{\mu}(x).
$$
If $\ell\notin \ZZ$, then either the first or the second term of the right hand side is a bounded factor times $P_{\ell}^{|\mu|}(x)$ which goes to $\infty$ as $x$ goes to $1^{-}$ by~\eqref{e:5}. Therefore $P_{\ell}^{\mu}(x)\rightarrow \infty$ as $x$ goes to $(-1)^{+}$, hence not a regular eigenfunction.

Therefore the only possibility left will be $\ell\in \ZZ, \mu\in \ZZ$. This leads to associated Legendre polynomials which can be checked explicitly. With the given boundary conditions, we get that $P_{\ell}^{\mu}(x)$ is a solution to the eigenvalue equation, if 
$$
\mu ,\ell\in \ZZ, \ 0< |\mu| \leq \ell.
$$
Note that $P_{\ell}^{\mu}(x)$ is a multiple of $P_{\ell}^{-\mu}(x)$ in this case, and to conform with the result from case 1, we get
\begin{equation}\label{e:ellmu1}
\begin{aligned}
&\mu=-1, -2, \dots,\\
& \ell=|\mu|, |\mu|+1, \dots.
\end{aligned}
\end{equation}
That is, the eigenvalues are given by 
\begin{equation}\label{e:eigenvalue1}
\lambda=(\tfrac{k\pi}{\beta}+j)(\tfrac{k\pi}{\beta}+j+1), \  k\in \NN^{+}, j\in \NN, \ \frac{k\pi}{\beta}\in \ZZ. 
\end{equation}
where the corresponding eigenfunction is given by
$$
u=P_{\frac{k\pi}{\beta}+j}^{-\frac{k\pi}{\beta}} \left(\cos (r)\right) \sin (\tfrac{k\pi}{\beta}\theta).
$$
As mentioned above, in this case the eigenfunction can also be written as
$$
u=P_{\frac{k\pi}{\beta}+j}^{\frac{k\pi}{\beta}} \left(\cos (r)\right) \sin (\tfrac{k\pi}{\beta}\theta).
$$

Combining~\eqref{e:eigenvalue2} and~\eqref{e:eigenvalue1} from the two cases, we get the result in the statement.
\end{proof} 

\subsection{Spherical triangles}

Now we consider the spherical triangle which is bounded by $\theta=0, \theta=\beta, r=\pi/2$ using the same coordinate as before, i.e. half of the spherical lunes discussed above. Note that when $\beta=\pi/2$ it is the equilateral triangle. With the same set up as in the spherical lune, the only change is that 
%
we are looking for
Legendre functions $P_{\ell}^{\mu}, Q_{\ell}^{\mu}$ with boundary conditions 
$$
P_{\ell}^{\mu}(1)=P_{\ell}^{\mu}(0)=0 \ \text{ or } \  Q_{\ell}^{\mu}(1)=Q_{\ell}^{\mu}(0)=0
$$
instead. 

\begin{proof}[Proof of Proposition~\ref{prop2.2}]
Since the eigenvalue equation~\eqref{e1} is invariant when changing $x$ to $-x$, an eigenfunction that satisfies $R(1)=R(0)=0$ can be extended to an odd eigenfunction $R$ on the whole interval $[-1,1]$ satisfying $R(1)=R(-1)=0$ with the same eigenvalue. Hence we only need to look at the eigenfunctions from the previous proposition, and find these ones satisfying an additional condition $R(0)=0$. 

We use the value of $P_{*}(x)$ at 0~\cite[(14.5.1)]{1}:
$$
P_{\ell}^{\mu}(0)=\frac{2^{\mu}\pi^{1/2}}{\Gamma(\frac{\ell-\mu}{2}+1)\Gamma(\frac{1}{2}-\frac{\ell+\mu}{2})}.
$$
From $P_{\ell}^{\mu}(0)=0$ we get either
\begin{equation}\label{e:01}
\frac{\ell-\mu}{2}+1=0,-1, -2, \dots,
\end{equation}
or
\begin{equation}\label{e:02}
\frac{1}{2}-\frac{\ell+\mu}{2}=0, -1, -2, \dots,
\end{equation}
should hold.




Combining with~\eqref{e:ellmu2} and~\eqref{e:ellmu1} we get 
\begin{align*}
& \mu<0,\\
& \ell=-\mu+1, -\mu+3, -\mu+5, \dots.
\end{align*}
The eigenvalues are contained in the following set
$$
\left\{
\left(\frac{k\pi}{\beta}+2j+1\right)\left(\frac{k\pi}{\beta}+2j+2\right): k\in \NN^{+}, j\in \NN
\right\}
$$
and the eigenfunctions are given by
$$
u=P_{\frac{k\pi}{\beta}+2j}^{-\frac{k\pi}{\beta}}\left(\cos(r)\right)\sin(\tfrac{k\pi}{\beta}\theta).
$$
We note here again when $\frac{k\pi}{\beta}\in \NN$ the eigenfunction can also be written as
$$
u=P_{\frac{k\pi}{\beta}+2j+1}^{\frac{k\pi}{\beta}}\left(\cos(r)\right)\sin(\tfrac{k\pi}{\beta}\theta).
$$
\end{proof}

For the equilateral triangle, we give the explicit form of the first two eigenvalues and eigenfunctions which will be used in the next section.
\begin{corollary}
For the equilateral triangle with $\beta=\frac{\pi}{2}$, the first eigenvalue is 12 and the corresponding eigenfunction with normalized $L^{2}$ norm is given by
\begin{equation}
u_{1}=\tilde C_{1}P_{3}^{2}(\cos(r))\sin (2\theta)=\sqrt{\frac{105}{2\pi}} \sin^{2}(r) \cos(r) \sin(2\theta).  \label{eq-u_1}
\end{equation}
The second eigenvalue is 30, and there are two corresponding normalized eigenfunctions given by
\begin{equation}
\begin{aligned}
&u_{2}^{(1)}=\tilde C_{2}P_{5}^{2}(\cos(r))\sin (2\theta)=\sqrt{\frac{1155}{8\pi}}(3\cos^5(r)-4\cos^3(r)+\cos(r))\sin (2\theta), \\
&u_{2}^{(2)}=\tilde C_{3}P_{5}^{4}(\cos(r))\sin (4\theta)=\sqrt{\frac{3465}{32\pi}}\cos(r)\sin^4(r)\sin(4\theta).  \label{eq-u_2}
\end{aligned}
\end{equation}
\end{corollary}

\section{Variation of Gap of  Spherical triangle with diameter $\frac{\pi}{2}$}
In this section we consider all spherical triangles with a fixed diameter $\frac{\pi}{2}$. It is not difficult to show that any such triangle can be moved on the sphere to have vertices  $(0,0)$, $(\frac{\pi}{2},0)$ and $(A,B)$ with $0<A<\frac{\pi}{2},\ 0<B<\frac{\pi}{2}$. 

Denote by $T$ the right triangle with vertices $(0,0)$, $(\frac{\pi}{2},0)$, $(\frac{\pi}{2},\frac{\pi}{2})$ and $T(t)$ the triangle with vertices $(0,0)$, $(\frac{\pi}{2},0)$ and $(\frac{\pi}{2}-bt,\frac{\pi}{2}-at)$ with $a^2+b^2=1, a \ge 0, \ b\ge 0$, see Figure~\ref{fig-T(t)}.

\begin{center}
	\begin{figure}[h]
	\includegraphics[scale=0.58]{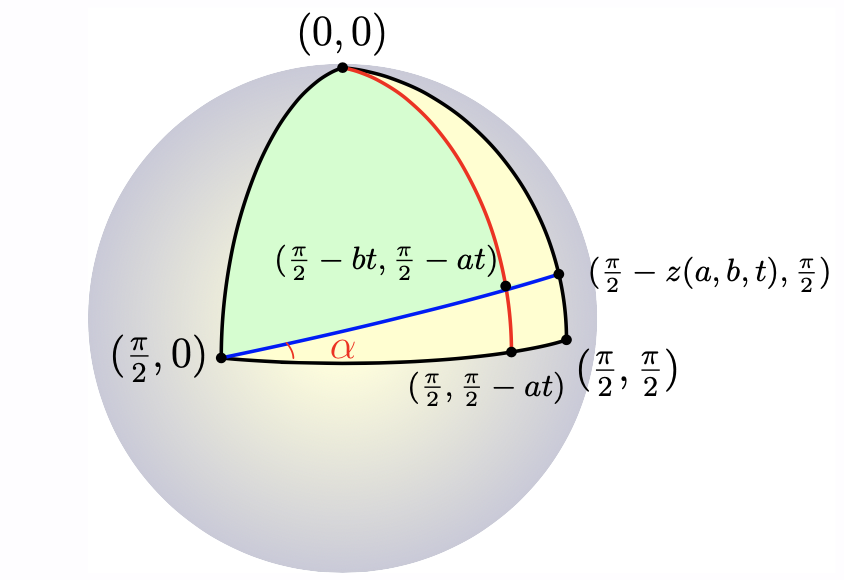}
	\caption{The deformed triangle $T(t)$ with three vertices at $(0,0), (\frac{\pi}{2}, 0), (\frac{\pi}{2}-bt, \frac{\pi}{2}-at)$.}
	\label{fig-T(t)}
	\end{figure}
\end{center}

We first construct a diffeomorphism $F_t$ which maps the triangle $T$  to $T(t)$.  To construct such a mapping, we first compute the function $l(\alpha,\theta)$ which gives the geodesic distance from the equator to the edge of the deformed triangle, see Figure~\ref{sphere-l}.
\begin{center}  
	\begin{figure}[!h]
\includegraphics[scale=0.6]{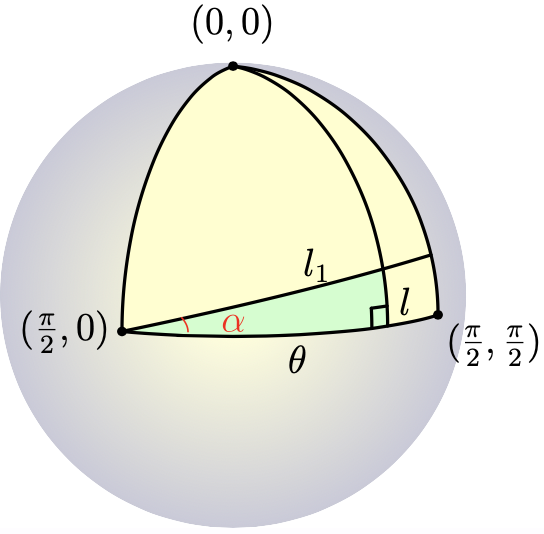}
 \caption{A spherical triangle with one side of length $\theta$ and two angles $\alpha, \pi/2$. The function $\ell(\alpha, \theta)$ computes the length of the side opposite to $\alpha$.}
 	\label{sphere-l}
 \end{figure}
\end{center}

For the spherical  triangle with length $\theta, l , l_1$, by the spherical  cosine law, 
\begin{align*}
\cos(l_1) = \cos(l)\cos(\theta), 
\end{align*}
and spherical law of sines
\begin{align*}
\sin(l_1) = \frac{\sin(l)}{\sin(\alpha)}
\end{align*}
we get 
\begin{align*}
\frac{\sin^2(l)}{\sin^2(\alpha)} + \cos^2(\theta)\cos^2(l) =1.
\end{align*}
Namely

\begin{align*}
(1-\sin^2(\alpha)\cos^2(\theta))\sin^2(l) = \sin^2(\alpha)\sin^2(\theta),
\end{align*}
so
\begin{align*}
\sin(l) = \pm\frac{\sin(\alpha)\sin(\theta)}{\sqrt{1-\sin^2(\alpha)\cos^2(\theta)}}.
\end{align*}
Since $\ell\leq \pi/2$, 
\begin{equation}\label{eleqn}
l(\alpha,\theta) = \arcsin\left( \frac{\sin(\alpha)\sin(\theta)}{\sqrt{1-\sin^2(\alpha)\cos^2(\theta)}}\right).
\end{equation}

Now let $z(a,b,t)$ be the distance between the vertex $(\frac{\pi}{2},\frac{\pi}{2})$ to the intersection of the edge of the deformed triangle and the $x=0$ plane. With the notation given in Figure~\ref{fig-T(t)}, we have  $z(a,b,t) = \alpha$. 

Since we have 
\begin{align*}
l\left(z(a,b,t),\tfrac{\pi}{2}-at\right) = bt,
\end{align*}
using \eqref{eleqn} this gives
\begin{align*}
\frac{\sin(\alpha)\sin(\theta)}{\sqrt{1-\sin^2(\alpha)\cos^2(\theta)}} = \sin(bt). 
\end{align*}
Solving  for $\sin \alpha$ gives
\begin{align*}
\sin(\alpha) = \frac{\sin(bt)}{\sqrt{\sin^2(\theta)+\sin^2(bt)\cos^2(\theta)}}.
\end{align*}
Hence
\begin{align*}
z(a,b,t) =\alpha&= \arcsin\left( \frac{\sin(bt)}{\sqrt{\sin^2(\frac{\pi}{2}-at)+\sin^2(bt)\cos^2(\frac{\pi}{2}-at)}} \right)\\
&=\arcsin\left( \frac{\sin(bt)}{\sqrt{\cos^2(at)+\sin^2(bt)\sin^2(at)}} \right).
\end{align*}

\subsection{Deformation map and the Laplacian}

We define the deformation map $F_t: T \to T(t)$ by 
\begin{align}
F_t(r,\theta) = \left(r-l\left(z(a,b,t),\theta-\frac{2a}{\pi}\theta t\right)\frac{2r}{\pi},\theta -\frac{2a}{\pi}\theta t\right) = (s,\psi).
\end{align}
With the computation above, we have
\begin{align*}
F_t(\tfrac{\pi}{2},\tfrac{\pi}{2}) = (\tfrac{\pi}{2}-bt,\frac{\pi}{2}-at).
\end{align*}
We also have
$$
F_{t}(0,0)=(0,0), \ F_{t}(\frac{\pi}{2}, 0)=(\frac{\pi}{2},0).
$$

We will need the following asymptotics. Since $
z(a,b,0) = 0, \ 
\frac{\partial}{\partial_{t}}z(a,b,0) = b
$, we have  $$z(a,b,t) = bt + O(t^2).$$
By \eqref{eleqn}, 
\begin{align*}
l(z,(1-A)\theta) = \arcsin\left( \frac{\sin(z)\sin((1-A)\theta)}{\sqrt{1-\sin^2(z)\cos^2((1-A)\theta)}}\right)
\end{align*}
where $A=\frac{2at}{\pi}$.  Then $
l|_{t=0} = 0, \  
\frac{\partial l}{\partial t}|_{t=0} =b\sin(\theta),
$
so \begin{equation}
l(z,(1-A)\theta) = b\sin(\theta) t + O(t^2).  \label{l-asymp}
\end{equation} Define
\begin{align*}
L:= \frac{\partial}{\partial\theta}[l(z,(1-A)\theta)].
\end{align*}
Then
\begin{align}
L = b\cos(\theta)t +O(t^2), \ \ \ 
\dd_\theta L = -b\sin(\theta)t + O(t^2).  \label{L-asymp}
\end{align}

To compute the variation of the Laplacian of the triangle $T(t)$, we fix the domain by pullback the round metric
\begin{align*}
g_{S} = dr^2 +\sin^2(r)d\theta
\end{align*}
on $T(t)$ with the diffeomorphism $F_t$ to $T$.  Note that when evaluating the pullback metric at $p\in T$, we evaluate the round metric at $F_{t}(p) \in T(t)$ so that 
\begin{align*}
g_t|_{p} = (F_{t}^*g_S)|_{p} = (dF_t|_p)^Tg_S|_{F_{t}(p)}dF_t|_p,
\end{align*}
where
\begin{align*}
g_S|_{F_t(p)} = \begin{pmatrix}
1 & 0 \\
0 & \sin^2(r(1-\frac{2l}{\pi}))
\end{pmatrix}
\end{align*}
and
\begin{align*}
dF_t =
\begin{pmatrix}
1-\frac{2}{\pi}l & -\frac{2r}{\pi}L \\
0 & (1-A)
\end{pmatrix}.
\end{align*}
Then
\begin{align*}
g_t = F_{t}^*g_S = \begin{pmatrix}
(1-\frac{2}{\pi}l)^2 & -\frac{2r}{\pi}(1-\frac{2}{\pi}l)L\\
-\frac{2r}{\pi}(1-\frac{2}{\pi}l)L & \frac{4r^2}{\pi^2}L^2 + (1-A)^2\sin^2(r(1-\frac{2l}{\pi}))
\end{pmatrix}
\end{align*}
and
\begin{align}\label{e:detgt}
\det(g_t) = (1-\tfrac{2}{\pi}l)^2(1-A)^2\sin^2(r(1-\tfrac{2l}{\pi})),
\end{align}
\begin{align*}
g_t^{-1}&=\begin{pmatrix}
\frac{\frac{4r^2}{\pi^2}L^2}{(1-\frac{2}{\pi}l)^2(1-A)^2}\csc^2(r(1-\frac{2l}{\pi})) + \frac{1}{(1-\frac{2}{\pi}l)^2} & \frac{\frac{2r}{\pi}L}{(1-\frac{2}{\pi}l)(1-A)^2}\csc^2(r(1-\frac{2l}{\pi})) \\
\frac{\frac{2r}{\pi}L}{(1-\frac{2}{\pi}l)(1-A)^2}\csc^2(r(1-\frac{2l}{\pi})) & \frac{1}{(1-A)^2}\csc^2(r(1-\frac{2l}{\pi}))
\end{pmatrix}.
\end{align*}
From this we can compute the Laplacian $\Delta_t$ of $g_t$ using the formula
\begin{align*}
\Delta f = \frac{1}{\sqrt{\deg(g)}}\dd_i[g^{ij}\sqrt{g}\dd_j f].
\end{align*}
We compute
\begin{align*}
\dd_r(g^{rr}\sqrt{g}\dd_r) 
&=\dd_r\left[\left(\frac{4r^2}{\pi^2}\frac{L^2}{(1-\frac{2}{\pi}l)(1-A)}\csc(r(1-\tfrac{2l}{\pi})) + \frac{(1-A)\sin(r(1-\frac{2l}{\pi}))}{(1-\frac{2}{\pi}l)}\right)\dd_r\right]\\
&=\left(\frac{4r^2}{\pi^2}\frac{L^2}{(1-\frac{2}{\pi}l)(1-A)}\csc(r(1-\tfrac{2l}{\pi})) + \frac{(1-A)\sin(r(1-\frac{2l}{\pi}))}{(1-\frac{2}{\pi}l)}\right)\dd_r^2\\
&\hspace{0.2 in} + \left((1-A)\cos(r(1-\tfrac{2l}{\pi}))-\frac{4r^2}{\pi^2}\frac{L^2}{(1-A)}\csc(r(1-\tfrac{2l}{\pi}))\cot(r(1-\tfrac{2l}{\pi}))\right.\\
&\hspace{0.2 in}\left. + \frac{8r}{\pi^2}\frac{L^2}{(1-\frac{2}{\pi}l)(1-A)}\csc(r(1-\tfrac{2l}{\pi}))  \right)\dd_r,
\end{align*}
and
\begin{align*}
\dd_r(g^{r\theta}\sqrt{\det g}\dd_\theta) &=\dd_r\left[\frac{\frac{2r}{\pi}L}{(1-A)}\csc(r(1-\tfrac{2l}{\pi}))\dd_\theta \right]\\
&=\frac{2r}{\pi}\frac{L}{1-A}\csc(r(1-\tfrac{2l}{\pi}))\dd_r\dd_\theta + \frac{2}{\pi}\frac{L}{(1-A)}\csc(r(1-\tfrac{2l}{\pi}))\dd_\theta\\
&\hspace{0.2in} - \frac{2r}{\pi}\frac{L(1-\frac{2l}{\pi})}{1-A}\csc(r(1-\tfrac{2l}{\pi}))\cot(r(1-\tfrac{2l}{\pi}))\dd_\theta,
\end{align*}
and
\begin{align*}
\dd_\theta[g^{\theta r}\sqrt{\deg g}\dd_r] &=\dd_\theta\left[\frac{\frac{2r}{\pi}L}{(1-A)}\csc(r(1-\tfrac{2l}{\pi}))\dd_r \right]\\
&=\frac{2r}{\pi}\frac{L}{(1-A)}\csc(r(1-\tfrac{2l}{\pi}))\dd_\theta\dd_r +\frac{2r}{\pi}\frac{\dd_\theta L}{1-A}\csc(r(1-\tfrac{2l}{\pi}))\dd_r \\
&\hspace{0.2 in} + \frac{4r^2L}{\pi^2}\frac{\dd_\theta L}{1-A}\csc(r(1-\tfrac{2l}{\pi}))\cot(r(1-\tfrac{2l}{\pi}))\dd_r,
\end{align*}
and
\begin{align*}
\dd_\theta[g^{\theta\theta}\sqrt{\det g}\dd_\theta] &=\dd_\theta\left[ \frac{(1-\frac{2}{\pi}l)}{(1-A)}\csc(r(1-\tfrac{2l}{\pi}))\dd_\theta \right]\\
&=-\frac{2}{\pi}\frac{L}{(1-A)}\csc(r(1-\tfrac{2l}{\pi}))\dd_\theta + \frac{(1-\frac{2}{\pi}l)}{(1-A)}\csc(r(1-\tfrac{2l}{\pi}))\dd_\theta^2 \\
&\hspace{0.2 in} + \frac{2rL}{\pi}\frac{(1-\frac{2}{\pi}l)}{(1-A)}\csc(r(1-\tfrac{2l}{\pi}))\cot(r(1-\tfrac{2l}{\pi}))\dd_\theta.
\end{align*}

Combining terms and using $l,L=O(t)$, we have
\begin{align*}
\Delta_t &= \frac{1}{(1-\frac{2}{\pi}l)^2}\dd_r^2 + \frac{1}{(1-\frac{2}{\pi}l)}\cot(r(1-\tfrac{2l}{\pi}))\dd_r+\frac{2r}{\pi}\frac{\dd_\theta L}{(1-\frac{2l}{\pi})(1-A)^2}\csc^2(r(1-\tfrac{2l}{\pi}))\dd_r \\
&\hspace{0.2 in} + \frac{4r}{\pi}\frac{L}{(1-\frac{2}{\pi}l)(1-A)^2}\csc^2(r(1-\tfrac{2l}{\pi}))\dd_r\dd_\theta + \frac{1}{(1-A)^2}\csc^2(r(1-\tfrac{2l}{\pi}))\dd_\theta^2 + O(t^2)
\end{align*}

Using the series expansions
\begin{align*}
\cot(r(1-\tfrac{2l}{\pi})) &= \cot(r) + \frac{2rl}{\pi}\csc^2(r) + O(t^2),\\
\csc^2(r(1-\tfrac{2l}{\pi})) &= \csc^2(r)+\frac{4rl}{\pi}\cot(r)\csc^2(r)+O(t^2),
\end{align*}
and
\begin{align*}
\frac{1}{(1-\frac{2}{\pi}l)^2} =  1+\frac{4}{\pi}l + O(t^2), \ \ \
\frac{1}{(1-A)^2} =  1+2A + O(t^2),
\end{align*}
and plugging in the first order term for $l$, $L$ and $\dd_\theta L$ from \eqref{l-asymp}, \eqref{L-asymp},  and $A=\frac{2at}{\pi}$, 
we obtain the following asymptotic formula.
\begin{lemma} The first order asymptotic expansion of the Laplacian of the deformed triangle $T(t)$ is given by
	\begin{align}
	\Delta_t &=\Delta_S+t L_1
+ O(t^2),
\end{align}
where $\Delta_S$ is the standard sphere Laplacian \eqref{lap} and
\begin{align}  \label{def-L_1}
L_1 &:= \frac{4}{\pi}b\sin(\theta)\dd_r^2 + \frac{2}{\pi}b\sin(\theta)\cot(r)\dd_r + \frac{4}{\pi}br\cos(\theta)\csc^2(r)\dd_r\dd_\theta \\
&\hspace{0.2 in}+\frac{4}{\pi}br\sin(\theta)\cot(r)\csc^2(r)\dd_\theta^2 + \frac{4}{\pi}a\csc^2(r)\dd_\theta^2. \nonumber
\end{align} 
\end{lemma}

\subsection{Perturbation of eigenvalues}
Let $u_1$ be the eigenfunction for $\lambda_1$ on the equilateral triangle with the round metric with unit norm (for explicit form see \eqref{eq-u_1}).  Let $f_1(t)$ and $\lambda_1(t)$ be the first eigenfunction and eigenvalue for $T(t)$.  
By the simplicity of $\lambda_1(t)$, it is differentiable.  Then
\begin{align*}
\lambda_1(t) &= \lambda_1+t\dot{\lambda_1}+O(t^2)\\
f_1(t) &= u_1 +t\dot{f_1} + O(t^2).
\end{align*}
Denote by $\langle, \rangle_{T}$ the inner product over the equilateral triangle $T$ with round metric, for small $t$ we have
\begin{align*}
\lambda_1(t)\langle f_1(t),f_1(t)\rangle_{T} &= -\langle \Delta_tf_1(t),f_1(t)\rangle_T\\
&= -\langle (\Delta_{S^2}+tL_1)[u_1+t\dot{f_1}],u_1+t\dot{f_1}\rangle_T  + O(t^2)\\
&= \lambda_1\|u_1\|^2 +2t\lambda_1\langle u_1,\dot{f_1}\rangle_{T} -t\langle L_1u_1,u_1\rangle_T + O(t^2).
\end{align*}
On the other hand
\begin{align*}
\lambda_1(t)\langle f_1(t),f_1(t)\rangle_T&= (\lambda_1+t\dot{\lambda_1})\langle u_1+t\dot{f_1},u_1+t\dot{f_1}\rangle_T  + O(t^2)\\
&=\lambda_1\|u_1\|^2+2t\lambda_1\langle u_1,\dot{f_1}\rangle_T+t\dot{\lambda_1}\|u_1\|^2  + O(t^2).
\end{align*}
Since $u_1$ have unit $L^2$ norm, we get
\begin{equation}\label{lambda1}
\dot{\lambda_1} = -\langle L_1u_1,u_1\rangle_T.
\end{equation}

Under the deformation, the relation between the integrals is 
\begin{align*}
\int_{T(t)} f \sin (r)drd\theta =\int_{T(t)=F_t(T)} f \sqrt{\det(g_S)} = \int_T F_t^{*}[f] \sqrt{\det(F^*_t g_S}) = (1-A)\int_T F_t^{*}[f] (1-\tfrac{2}{\pi}l)\sin(r(1-\tfrac{2l}{\pi}))
\end{align*}
where the second equality comes from~\eqref{e:detgt}.
Therefore, using \eqref{l-asymp} and definition of $A$, we have
\begin{align*}
\int_{T(t)}f_1f_2\sin(r)drd\theta &= (1-A)\int_T F_{t}^*[f_1f_2] (1-\tfrac{2}{\pi}l)\sin(r(1-\tfrac{2l}{\pi}))drd\theta\\
&=(1-A)\int_T F_{t}^*[f_1f_2] (1-\tfrac{2}{\pi}l)(\sin(r)-\tfrac{2rl}{\pi}\cos(r))drd\theta +O(t^2)\\
&= \int_T F_{t}^*[f_1f_2]\sin(r)drd\theta-A\int_TF_{t}^*[f_1f_2]\sin(r)drd\theta\\
&-\frac{2}{\pi} l \left(\int_TF_{t}^*[f_1f_2]\sin(r)drd\theta +\int_TF_{t}^*[f_1f_2]r\cos(r)drd\theta \right) +O(t^2)\\
&:= \int_T F_{t}^*[f_1f_2]\sin(r)drd\theta + tZ+O(t^2).
\end{align*}

If $f_1$ and $f_2$ are eigenfunctions for the first two Dirichlet eigenvalues on $T(t)$ (with round metric) then by orthogonality we have that
\begin{align}
\int_T F_{t}^*[f_1f_2]\sin(r)drd\theta =-tZ+O(t^2).  \label{int-asym}
\end{align}

\begin{lemma}  \label{gap-var}
Let $u_{1}, u_{2}$ be eigenfunctions for $T$ with unit $L^{2}$ norm corresponding to the first two eigenvalues $\lambda_{1}, \lambda_{2}$.
Suppose that for any $a,b\geq 0$,
with respect to the linear order operator $L_1$ defined in \eqref{def-L_1}, 
\begin{align*}
\int_T u_2L_1 u_2- u_1L_1u_1 < 0.
\end{align*}
Then the equilateral triangle $T$ is a strict local minimum for the gap function among all spherical triangles with diameter $\frac{\pi}{2}$.
\end{lemma}

\begin{proof}
 Let $f_1$ and $f_2$ be eigenfunctions for the first two Dirichlet eigenvalues of the deformed triangle $T(t)$.  Here we will integrate over the equilateral triangle $T$ with the round metric.  Since $\Delta_tf_2 =  -\lambda(t)f_2$ is pointwise, it still satisfies the eigenvalue equation after pullback. And up to first order, $F_{t}^*[f_2] = f_2+O(t)$. By abuse of notation, $f_i = F_{t}^*[f_i]$.  Then define
\begin{align*}
\eps(t) := \frac{-\int_{T }u_1 f_2}{\int_{T} u_1f_1}.
\end{align*}
Since $T(0) =T$ and $F_0 = \id$, the expansion $f_1=u_1+t\frac{d}{dt}|_{t=0}f_1+ O(t^2)$ implies ${\int_{T} u_1f_1}=1+O(t)$. Then by \eqref{int-asym}, 
\begin{align*}
\eps(t) = \frac{\int_{T}f_2(f_1-u_1)}{\int_{T} u_1f_1} +tZ +O(t^2).
\end{align*}
Using the same expansion $f_1=u_1+t\frac{d}{dt}|_{t=0}f_1+ O(t^2)$ it implies that $\eps(t) = O(t)$, for small $t$.  By definition of $\eps (t)$, we have 
\begin{align*}
\int_{T} (f_2+\eps f_1)u_1 = 0.
\end{align*}
So we can use $f_2 +\eps f_1$ as a test function for $\lambda_2$,
\begin{align*}
\lambda_2 \leq \frac{-\int_T (f_2+\eps f_1)\Delta_{S^2}(f_2+\eps f_1)}{\int_T (f_2+\eps f_1)^2}.
\end{align*}
Using the asymptotic $\Delta_{S^2} = \Delta_t -tL_1+O(t^2)$, 
\begin{align*}
\frac{-\int_T (f_2+\eps f_1)\Delta_{S^2}(f_2+\eps f_1)}{\int_T (f_2+\eps f_1)^2} &= \frac{-\int_T (f_2+\eps f_1)(\Delta_t-tL_1)(f_2+\eps f_1)+O(t^2)}{\int_T (f_2+\eps f_1)^2}\\
&=\frac{\lambda_2(t)\int_Tf_2^2 + t\int_T f_2L_1f_2 +O(t^2)}{\int_T (f_2+\eps f_1)^2}.
\end{align*}
Since $\int_T f_1f_2 = O(t)$ and $\int_T f_2^2 = 1+O(t)$, we have
\begin{align*}
\lambda_2 \leq \lambda_2(t) +t\int_Tf_2L_1f_2+O(t^2).
\end{align*}
Therefore, combining with \eqref{lambda1} gives
\begin{align*}
\lambda_2-\lambda_1 \leq \lambda_2(t)-\lambda_1(t) + t\left( \int_T f_2L_1f_2 - \int_T u_1L_1 u_1 \right) +O(t^2).
\end{align*}
Using the asymptotics of $f_2$ once more, we have
\begin{align*}
\lambda_2-\lambda_1 \leq \lambda_2(t)-\lambda_1(t) + t\left( \int_T u_2L_1 u_2 - \int_Tu_1L_1u_1 \right) + O(t^2).
\end{align*}
Hence, with the assumption
$$\int_T u_2L_1 u_2- u_1L_1u_1 < 0,$$ 
for small $t$ we have $\lambda_2-\lambda_1 < \lambda_2(t)-\lambda_1(t)$.
\end{proof}


\subsection{Computation for $\int_T u_1 L_1 u_1$}
Using  the explicit expressions for $u_1$ \eqref{eq-u_1} and $L_1$ \eqref{def-L_1}, we have
\begin{align*}
\int_T u_1 L_1 u_1\sqrt{\det g_{S}} =& \frac{4}{\pi}b\int_T u_1 \sin(\theta)\dd_r^2[u_1]\sin(r)drd\theta \hspace{1.35 in} (\I)\\
&+\frac{2}{\pi}b\int_T u_1\sin(\theta)\cot(r)\dd_r[u_1]\sin(r)drd\theta \hspace{0.75 in} (\II)\\
&+\frac{4}{\pi}b \int_T u_1 r\cos(\theta)\csc^2(r) \dd_r\dd_\theta[u_1]\sin(r)drd\theta \hspace{0.4 in} (\III) \\
&+\frac{4}{\pi}b\int_T u_1 r\sin(\theta)\cot(r)\csc^2(r)\dd_\theta^2[u_1]\sin(r)drd\theta \quad (\IV)\\
&+\frac{4}{\pi}a\int_T u_1 \csc^2(r)\dd_\theta^2[u_1]\sin(r)drd\theta. \hspace{1.05 in} (\V)
\end{align*}
Denote $C_1 = \frac{105}{2\pi}$. Now calculating each term, we have for term I
\begin{align*}
\frac{4}{\pi}bC_1\int_0^{\frac{\pi}{2}}\sin^2(2\theta)\sin(\theta)d\theta \int_0^{\frac{\pi}{2}}\sin^3(r)\cos(r)\dd_r^2[\sin^2(r)\cos(r)]dr &= -bC_1\frac{1408}{1575\pi}.
\end{align*}
For term II
\begin{align*}
\frac{2}{\pi}bC_1\int_0^{\frac{\pi}{2}}\sin^2(2\theta)\sin(\theta)d\theta\int_0^{\frac{\pi}{2}}\sin^2\cos^2(r)\dd_r[\sin^2(r)\cos(r)]dr= bC_1 \frac{64}{1575\pi}.
\end{align*}
For term III
\begin{align*}
\frac{4}{\pi}bC_1\int_0^{\frac{\pi}{2}}\sin(2\theta)\cos(\theta)\dd_\theta[\sin(2\theta)]d\theta \int_0^{\frac{\pi}{2}}r\sin(r)\cos(r)\dd_r[\sin^2(r)\cos(r)]dr =bC_1\left(\frac{16}{450} - \frac{448}{3375\pi}\right).
\end{align*}
For term IV
\begin{align*}
\frac{4b}{\pi}C_1 \int_0^{\frac{\pi}{2}}\sin(2\theta)\sin(\theta)\dd_\theta^2[\sin(2\theta)]d\theta \int_0^{\frac{\pi}{2}}r\cos^3(r)\sin^2(r)dr &= bC_1\left(\frac{3328}{3375\pi}-\frac{128}{225}\right).
\end{align*}
For term V
\begin{align*}
\frac{4}{\pi}a\int_0^{\frac{\pi}{2}}\sin(2\theta)\dd_\theta^2[\sin(2\theta)]d\theta\int_0^{\frac{\pi}{2}}\sin^3(r)\cos^2(r)dr = -\frac{8a}{15}C_1.
\end{align*}
Combining, we obtain
\begin{align}
\int_Tu_1L_1u_1dA_{S^2}&=bC_1\left(-\frac{1408}{1575\pi}+\frac{64}{1575\pi}-\frac{448}{3375\pi}+\frac{3328}{3375\pi} + \frac{16}{450}-\frac{128}{225}\right) - \frac{8}{15}aC_1  \nonumber\\
&=-b\frac{28}{\pi}-a\frac{28}{\pi}.
\end{align}

\subsection{Computation  for $\int_T u_2 L_1 u_2$}
By linearity, the second eigenfunction is of the form
\begin{align*}
u_2:= pu^{(1)}_2+qu^{(2)}_2
\end{align*}
with $p^2+q^2=1$ and $u_{2}^{(1)}, u_{2}^{(2)}$ given in~\eqref{eq-u_2}.
Then
\begin{align}  \label{u2L1u2}
\int_T u_2 L_1 u_2 &= p^2\int_T u^{(1)}_2(L_1u^{(1)}_2) + pq\int_T u^{(2)}_2(L_1 u^{(1)}_2) + pq\int_Tu^{(1)}_2(L_1u^{(2)}_2)+q^2\int_T u^{(2)}_2(L_1u^{(2)}_2) \\
&=p^2(b\frac{77}{\pi}+a\frac{44}{\pi}) - pqb\frac{22\sqrt{3}}{\pi} +q^2(b\frac{55}{\pi}+a\frac{88}{\pi}) \nonumber
\end{align}
The details of the computation are shown in the appendix.

Define
\begin{align*}
I &:= -\int_T u_2 (L_1 u_2) + \int_T u_1(L_1u_1) \\
&=b\left(p^2\frac{77}{\pi}-pq\frac{22\sqrt{3}}{\pi}+q^2\frac{55}{\pi}-\frac{28}{\pi}\right) + a\left(p^2\frac{44}{\pi}+q^2\frac{88}{\pi}-\frac{28}{\pi}\right).
\end{align*}

Using $p=\cos(z)$, $q=\sin(z)$ and $a = \sqrt{1-b^2}$,

\begin{align*}
I &= b\left( \frac{27}{\pi}+\cos^2(z)\frac{22}{\pi}-\cos(z)\sin(z)\frac{22\sqrt{3}}{\pi}\right)+\sqrt{1-b^2}\left(\frac{16}{\pi}+\sin^2(z)\frac{44}{\pi}\right)
\end{align*}

To find the minimum over $0\leq z \leq 2\pi$ and $0\leq b \leq 1$, notice that the function $f(x) = Ax+B\sqrt{1-x^2}$ has $f''(x) = -\frac{B}{(1-x^2)^{\frac{3}{2}}} <0$ for $B>0$. Hence for each fixed $z$, any interior critical point of $I$ will be a maximum so the minimum must occur at the boundary ($b=0$ or $b=1$).  The minimum of $\frac{27}{\pi}+\cos^2(z)\frac{22}{\pi}-\cos(z)\sin(z)\frac{22\sqrt{3}}{\pi}$ is $\frac{16}{\pi}$, which is also the minimum of $\frac{16}{\pi}+\sin^2(z)\frac{44}{\pi}$, hence the minimum value is $I=\frac{16}{\pi}$.

Combine with Lemma~\ref{gap-var} this finishes the proof of Theorem~\ref{gap-tri}. 

\begin{remark}
	Note that when $a=1, \ b=0$ or $a=0, b=1$, the  variation is along one side of the equilateral spherical triangle. In both cases the minimum is $\frac{16}{\pi}$. In this case the gap is explicitly given in \eqref{beta-gap}. Namely $\Gamma (T(t)) = \frac{4\pi}{(\frac{\pi}{2}-t)}+10$. Hence $\frac{d}{dt}\Gamma (T(t))|_{t=0} = \frac{16}{\pi}$. 	So the above computation matches up with this direct computation. 	
%
	
\end{remark}

%
%
%
%
%

\appendix

\section{Details for the computation of  $\int_T u_2 L_1 u_2$}
We include here the detailed computation for~\eqref{u2L1u2} which is used for the variation of $\lambda_2(t)$. Recall the second eigenfunctions $u_{2}^{(1)}, u_{2}^{(2)}$ are given in~\eqref{eq-u_2}. Denote $C_2 = \frac{1155}{8\pi}, \ C_3 = \frac{3465}{32\pi}$. 


We first compute the $p^2$ term in \eqref{u2L1u2}:

\begin{align*}
\int_T u^{(1)}_2 L_1 u^{(1)}_2dA_{S^2} =& \frac{4}{\pi}b\int_T u^{(1)}_2 \sin(\theta)\dd_r^2[u^{(1)}_2]\sin(r)drd\theta \hspace{1.35 in} (\text{I})\\
&+\frac{2}{\pi}b\int_T u^{(1)}_2\sin(\theta)\cot(r)\dd_r[u^{(1)}_2]\sin(r)drd\theta \hspace{0.75 in} (\text{II})\\
&+\frac{4}{\pi}b \int_T u^{(1)}_2 r\cos(\theta)\csc^2(r) \dd_r\dd_\theta[u^{(1)}_2]\sin(r)drd\theta \hspace{0.4 in} (\III) \\
&+\frac{4}{\pi}b\int_T u^{(1)}_2 r\sin(\theta)\cot(r)\csc^2(r)\dd_\theta^2[u^{(1)}_2]\sin(r)drd\theta \quad (\IV)\\
&+\frac{4}{\pi}a\int_T u^{(1)}_2 \csc^2(r)\dd_\theta^2[u^{(1)}_2]\sin(r)drd\theta. \hspace{1.05 in} (\V)
\end{align*}

For term I,
\begin{align*}
\frac{4}{\pi}bC_2 &\int_0^{\frac{\pi}{2}}\sin^2(2\theta)\sin(\theta)d\theta\int_0^{\frac{\pi}{2}}(3\cos^5(r)-4\cos^3(r)+\cos(r))\dd_r^2[ (3\cos^5(r)-4\cos^3(r)+\cos(r))] \sin(r)dr\\
&=-bC_2\frac{6656}{5775\pi}.
\end{align*}

For term II,
\begin{align*}
\frac{2}{\pi}bC_2&\int_0^{\frac{\pi}{2}}\sin^2(2\theta)\sin(\theta)d\theta\int_0^{\frac{\pi}{2}} (3\cos^6(r)-4\cos^4(r)+\cos^2(r))\dd_r[ (3\cos^5(r)-4\cos^3(r)+\cos(r))]dr\\
&=bC_2\frac{256}{17325\pi}.
\end{align*}

For term III
\begin{align*}
\frac{2}{\pi}bC_2&\int_0^{\frac{\pi}{2}}\sin(2\theta)\cos(\theta)\dd_\theta[\sin(2\theta)]d\theta \int_0^{\frac{\pi}{2}}r\csc(r)\dd_r[(3\cos^5(r)-4\cos^3(r)+\cos(r))^2]dr\\
&= bC_2\left( \frac{8}{225}-\frac{2816}{23625\pi}\right).
\end{align*}

For term IV
\begin{align*}
\frac{4}{\pi}bC_2&\int_0^{\frac{\pi}{2}}\sin(2\theta)\sin(\theta)\dd_\theta^2[\sin(2\theta)]d\theta \int_0^{\frac{\pi}{2}}r(3\cos^5(r)-4\cos^3(r)+\cos(r))^2\cot(r)\csc(r)dr\\
 &= bC_2\left(\frac{29696}{23625\pi}-\frac{128}{225}\right).
\end{align*}

For term V

\begin{align*}
\frac{4}{\pi}aC_2&\int_0^{\frac{\pi}{2}}\sin(2\theta)\dd_\theta^2[\sin(2\theta)]d\theta\int_0^{\frac{\pi}{2}}\csc(r)(3\cos^5(r)-4\cos^3(r)+\cos(r))^2dr\\
&=-aC_2\frac{32}{105}
\end{align*}

Combining
\begin{equation}\label{e:pp}
\begin{aligned}
\int_T u^{(1)}_2L_1u^{(1)}_2dA_{S^2}&=bC_2\left(-\frac{6656}{5775\pi}+\frac{256}{17325\pi} -\frac{2816}{23625\pi}+\frac{29696}{23625\pi} +\frac{8}{225}-\frac{128}{225} \right) -aC_2\frac{32}{105}\\
&=-b\frac{77}{\pi}-a\frac{44}{\pi}
\end{aligned}
\end{equation}

We then compute the first $pq$ term in \eqref{u2L1u2}:

\begin{align*}
\int_T u^{(2)}_2 L_1 u^{(1)}_2dA_{S^2} =& \frac{4}{\pi}b\int_T u^{(2)}_2 \sin(\theta)\dd_r^2[u^{(1)}_2]\sin(r)drd\theta \hspace{1.35 in} (\I)\\
&+\frac{2}{\pi}b\int_T u^{(2)}_2\sin(\theta)\cot(r)\dd_r[u^{(1)}_2]\sin(r)drd\theta \hspace{0.75 in} (\II)\\
&+\frac{4}{\pi}b \int_T u^{(2)}_2 r\cos(\theta)\csc^2(r) \dd_r\dd_\theta[u^{(1)}_2]\sin(r)drd\theta \hspace{0.4 in} (\III) \\
&+\frac{4}{\pi}b\int_T u^{(2)}_2 r\sin(\theta)\cot(r)\csc^2(r)\dd_\theta^2[u^{(1)}_2]\sin(r)drd\theta \quad (\IV)\\
&+\frac{4}{\pi}a\int_T u^{(2)}_2 \csc^2(r)\dd_\theta^2[u^{(1)}_2]\sin(r)drd\theta. \hspace{1.05 in} (\V)
\end{align*}

For term I,
\begin{align*}
\frac{4}{\pi}b\sqrt{C_2C_3}&\int_0^{\frac{\pi}{2}}\sin(4\theta)\sin(2\theta)\sin(\theta)d\theta\int_0^{\frac{\pi}{2}}\cos(r)\sin^5(r)\dd_r^2[(3\cos^5(r)-4\cos^3(r)+\cos(r))]dr\\
&=b\sqrt{C_2C_3}\frac{8192}{40425\pi}.
\end{align*}

For term II,
\begin{align*}
\frac{2}{\pi}b\sqrt{C_2C_3}&\int_0^{\frac{\pi}{2}}\sin(4\theta)\sin(2\theta)\sin(\theta)d\theta\int_0^{\frac{\pi}{2}}\cos^2(r)\sin^4(r)\dd_r[(3\cos^5(r)-4\cos^3(r)+\cos(r))]dr\\
&=-b\sqrt{C_2C_3}\frac{2048}{121275\pi}.
\end{align*}
For term III

\begin{align*}
\frac{4}{\pi}b\sqrt{C_2C_3}&\int_0^{\frac{\pi}{2}}\sin(4\theta)\cos(\theta)\dd_\theta[\sin(2\theta)]d\theta\int_0^{\frac{\pi}{2}}r\cos(r)\sin^3(r)\dd_r[ (3\cos^5(r)-4\cos^3(r)+\cos(r))]dr\\
& =b\sqrt{C_2C_3}\left( \frac{1936}{11025}-\frac{833536}{3472875\pi} \right)
\end{align*}

For term IV,
\begin{align*}
\frac{4}{\pi}b\sqrt{C_2C_3}&\int_0^{\frac{\pi}{2}} \sin(4\theta)\sin(\theta)\dd_\theta^2[\sin(2\theta)]d\theta \int_0^{\frac{\pi}{2}}r\sin^2(r)(3\cos^7(r)-4\cos^5(r)+\cos^3(r))dr\\
&= b\sqrt{C_2C_3}\left( \frac{188416}{3472875\pi} -\frac{256}{11025}\right).
\end{align*}

For Term V, 
\begin{align*}
\frac{4}{\pi}a\sqrt{C_2C_3}&\int_0^{\frac{\pi}{2}}\sin(4\theta)\dd_\theta^2[\sin(2\theta)]d\theta \int_0^{\frac{\pi}{2}}\sin^3(r)(3\cos^6(r)-4\cos^4(r)+\cos^2(r))dr\\
&=0
\end{align*}

Combining to get
\begin{equation}
\begin{aligned}\label{e:pq1}
\int_T u^{(2)}_2L_1u^{(1)}_2dA_{S^2}&=b\sqrt{C_2C_3}\left(\frac{8192}{40425\pi}-\frac{2048}{121275\pi}-\frac{833536}{3472875\pi}+\frac{188416}{3472875\pi}+\frac{1936}{11025}-\frac{256}{11025}\right)\\
&=b\frac{11\sqrt{3}}{\pi}
\end{aligned}
\end{equation}

Next is the second $pq$ term in \eqref{u2L1u2}:

\begin{align*}
\int_T u^{(1)}_2 L_1 u^{(2)}_2dA_{S^2} =& \frac{4}{\pi}b\int_T u^{(1)}_2 \sin(\theta)\dd_r^2[u^{(2)}_2]\sin(r)drd\theta \hspace{1.35 in} (\I)\\
&+\frac{2}{\pi}b\int_T u^{(1)}_2\sin(\theta)\cot(r)\dd_r[u^{(2)}_2]\sin(r)drd\theta \hspace{0.75 in} (\II)\\
&+\frac{4}{\pi}b \int_T u^{(1)}_2 r\cos(\theta)\csc^2(r) \dd_r\dd_\theta[u^{(2)}_2]\sin(r)drd\theta \hspace{0.4 in} (\III) \\
&+\frac{4}{\pi}b\int_T u^{(1)}_2 r\sin(\theta)\cot(r)\csc^2(r)\dd_\theta^2[u^{(2)}_2]\sin(r)drd\theta \quad (\IV)\\
&+\frac{4}{\pi}a\int_T u^{(1)}_2 \csc^2(r)\dd_\theta^2[u^{(2)}_2]\sin(r)drd\theta. \hspace{1.05 in} (\V)
\end{align*}

For term I,

\begin{align*}
\frac{4}{\pi}b\sqrt{C_2C_3}&\int_0^{\frac{\pi}{2}}\sin(2\theta)\sin(\theta)\sin(4\theta)d\theta \int_0^{\frac{\pi}{2}}(3\cos^5(r)-4\cos^3(r)+\cos(r))\dd_r^2[\cos(r)\sin^4(r)]\sin(r)dr \\
&= b\sqrt{C_2C_3}\frac{2048}{14553\pi}
\end{align*}

For term II,
\begin{align*}
\frac{2}{\pi}b\sqrt{C_2C_3}&\int_0^{\frac{\pi}{2}}\sin(2\theta)\sin(\theta)\sin(4\theta)d\theta \int_0^{\frac{\pi}{2}}(3\cos^6(r)-4\cos^4(r)+\cos^2(r))\dd_r[\cos(r)\sin^4(r) ]dr\\ 
&= b\sqrt{C_2C_3}\frac{1024}{72765\pi}
\end{align*}

For term III,
\begin{align*}
\frac{4}{\pi}b\sqrt{C_2C_3}&\int_0^{\frac{\pi}{2}}\sin(2\theta)\cos(\theta)\dd_\theta[\sin(4\theta)]d\theta\int_0^{\frac{\pi}{2}}r(3\cos^5(r)-4\cos^3(r)+\cos(r))\csc(r)\dd_r[\cos(r)\sin^4(r)]dr\\
&=b\sqrt{C_2C_3}\left(\frac{2704}{11025}-\frac{1291264}{3472875\pi}\right)
\end{align*}

For term IV
\begin{align*}
\frac{4}{\pi}b\sqrt{C_2C_3}&\int_0^{\frac{\pi}{2}}\sin(2\theta)\sin(\theta)\dd_\theta^2[\sin(4\theta)]d\theta\int_0^{\frac{\pi}{2}}r(3\cos^7(r)-4\cos^5(r)+\cos^3(r))\sin^2(r)dr\\
&=b\sqrt{C_2C_3}\left( \frac{753664}{3472875\pi} - \frac{1024}{11025}\right)
\end{align*}

For term V
\begin{align*}
\frac{4}{\pi}a\sqrt{C_2C_3}&\int_0^{\frac{\pi}{2}}\sin(2\theta)\dd_\theta^2[\sin(4\theta)]d\theta\int_0^{\frac{\pi}{2}}(3\cos^6(r)-4\cos^4(r)+\cos^2(r))\sin^3(r)dr \\
&=0
\end{align*}

Combining to get
\begin{equation}\label{e:pq2}
\begin{aligned}
\int_Tu^{(1)}_2L_1u^{(2)}_2dA_{S^2}&=b\sqrt{C_2C_3}\left(\frac{2048}{14553\pi}+\frac{1024}{72765\pi}-\frac{1291264}{3472875\pi}+\frac{753664}{3472875\pi}+\frac{2704}{11025}-\frac{1024}{11025}  \right)\\
&=b\sqrt{C_2C_3}\frac{16}{105}
\end{aligned}
\end{equation}

Last is the $q^2$ term in \eqref{u2L1u2}:
\begin{align*}
\int_T u^{(2)}_2 L_1 u^{(2)}_2dA_{S^2} =& \frac{4}{\pi}b\int_T u^{(2)}_2 \sin(\theta)\dd_r^2[u^{(2)}_2]\sin(r)drd\theta \hspace{1.35 in} (\I)\\
&+\frac{2}{\pi}b\int_T u^{(2)}_2\sin(\theta)\cot(r)\dd_r[u^{(2)}_2]\sin(r)drd\theta \hspace{0.75 in} (\II)\\
&+\frac{4}{\pi}b \int_T u^{(2)}_2 r\cos(\theta)\csc^2(r) \dd_r\dd_\theta[u^{(2)}_2]\sin(r)drd\theta \hspace{0.4 in} (\III) \\
&+\frac{4}{\pi}b\int_T u^{(2)}_2 r\sin(\theta)\cot(r)\csc^2(r)\dd_\theta^2[u^{(2)}_2]\sin(r)drd\theta \quad (\IV)\\
&+\frac{4}{\pi}a\int_T u^{(2)}_2 \csc^2(r)\dd_\theta^2[u^{(2)}_2]\sin(r)drd\theta. \hspace{1.05 in} (\V)
\end{align*}

For term I,
\begin{align*}
\frac{4}{\pi}bC_3&\int_0^{\frac{\pi}{2}}\sin^2(4\theta)\sin(\theta)d\theta\int_0^{\frac{\pi}{2}}\cos(r)\sin^5(r)\dd_r^2[\cos(r)\sin^4(r)]dr\\
&=-bC_3\frac{139264}{218295\pi}
\end{align*}

For term II,
\begin{align*}
\frac{2}{\pi}bC_3&\int_0^{\frac{\pi}{2}}\sin^2(4\theta)\sin(\theta)d\theta\int_0^{\frac{\pi}{2}}\cos^2(r)\sin^4(r)\dd_r[\cos(r)\sin^4(r)]dr\\
&=bC_3\frac{4096}{218295\pi}
\end{align*}

For term III,
\begin{align*}
\frac{4}{\pi}bC_3&\int_0^{\frac{\pi}{2}}\sin(4\theta)\cos(\theta)\dd_\theta[\sin(4\theta)]d\theta \int_0^{\frac{\pi}{2}}r\cos(r)\sin^3(r)\dd_r[\cos(r)\sin^4(r)]dr\\
 & = bC_3\left(\frac{32}{3969}-\frac{45056}{1250235\pi} \right)
\end{align*}

For term IV,
\begin{align*}
\frac{4}{\pi}bC_3&\int_0^{\frac{\pi}{2}}\sin(4\theta)\sin(\theta)\dd_\theta^2[\sin(4\theta)]d\theta \int_0^{\frac{\pi}{2}}r\cos^3(r)\sin^6(r)dr\\
&=bC_3\left(\frac{163840}{250047\pi}-\frac{2048}{3969}\right)
\end{align*}

For term V,
\begin{align*}
\frac{4}{\pi}aC_3&\int_0^{\frac{\pi}{2}}\sin(4\theta)\dd_\theta^2[\sin(4\theta)]d\theta\int_0^{\frac{\pi}{2}}\cos^2(r)\sin^7(r)dr \\&=-aC_3\frac{16^2}{315}
\end{align*}

Combining to get
\begin{equation}\label{e:qq}
\begin{aligned}
\int_T u^{(2)}_2L_1u^{(2)}_2dA_{S^2}&=bC_3\left(-\frac{139264}{218295\pi}+\frac{4096}{218295\pi}-\frac{45056}{1250235\pi}+\frac{163840}{250047\pi}+\frac{32}{3969}-\frac{2048}{3969}\right)-aC_3\frac{16^2}{315}\\
&=-b\frac{55}{\pi}-a\frac{88}{\pi}
\end{aligned}
\end{equation}

Combining the results~\eqref{e:pp},~\eqref{e:pq1},~\eqref{e:pq2} and~\eqref{e:qq} we get~\eqref{u2L1u2}.


\begin{bibdiv}
\begin{biblist}

	


\bib{AbramowitzStegun}{book} {
     TITLE = {Handbook of mathematical functions with formulas, graphs, and
              mathematical tables},
    EDITOR = {Abramowitz, Milton}
    Editor={Stegun, Irene A.},
      NOTE = {Reprint of the 1972 edition},
 PUBLISHER = {Dover Publications, Inc., New York},
      YEAR = {1992},
     PAGES = {xiv+1046},
      ISBN = {0-486-61272-4},
   MRCLASS = {00A20 (00A22 33-00)},
  MRNUMBER = {1225604},
}

\bibitem{fundamental}
Ben Andrews and Julie Clutterbuck.
\newblock Proof of the fundamental gap conjecture.
\newblock {\em J. Amer. Math. Soc.}, 24(3):899--916, 2011.

\bibitem{BCNSWW}
Theodora Bourni, Julie Clutterbuck, Xuan~Hien Nguyen, Alina Stancu, Guofang
Wei, and Valentina-Mira Wheeler.
\newblock The  vanishing of the fundamental gap of convex domains in
$\mathbb{H}^n$.
\newblock 
arXiv:2005.11784, 2020.

\bibitem{dai2018fundamental}
Xianzhe Dai, Shoo Seto, and Guofang Wei.
\newblock Fundamental gap estimate for convex domains on sphere-- the case $n=
2$.
\newblock {\em To appear in Comm. in Analysis and Geometry, arXiv:1803.01115},
2018.

\bibitem{daifundamental}
Xianzhe Dai, Shoo Seto, and Guofang Wei.
\newblock Fundamental gap comparison.
\newblock In {\em Surveys in Geometric Analysis 2018}, pages 1--16, 2019.

\bibitem{he2017fundamental}
Chenxu He and Guofang Wei.
\newblock Fundamental gap of convex domains in the spheres (with appendix {B}
by {Q}i {S}. {Z}hang).
\newblock {\em Amer. Journal of Math},   142,  no. 4, 1161-1192,  2020.

\bib{lu-rowlett}{article}{
	author={Lu, Zhiqin},
	author={Rowlett, Julie},
	title={The fundamental gap of simplices},
	journal={Comm. Math. Phys.},
	volume={319},
	date={2013},
	number={1},
	pages={111--145},
	issn={0010-3616},
	review={\MR{3034027}},
	doi={10.1007/s00220-013-1670-9},
}

	\bibitem{1}
	F.~W.~J. Olver, A.~B. {Olde Daalhuis}, D.~W. Lozier, B.~I. Schneider,
                R.~F. Boisvert, C.~W. Clark, B.~R. Miller, B.~V. Saunders,
                H.~S. Cohl, and M.~A. McClain, eds.
	\newblock{\it NIST Digital Library of Mathematical Functions},
	\newblock{http://dlmf.nist.gov/, Release 1.0.27 of 2020-06-15.}


\bibitem{seto2019sharp}
Shoo Seto, Lili Wang, and Guofang Wei.
\newblock Sharp fundamental gap estimate on convex domains of sphere.
\newblock {\em Journal of Differential Geometry}, 112(2):347--389, 2019.

\end{biblist}
\end{bibdiv}
\end{document}